\documentclass[reqno, a4paper, draft]{amsart}
\usepackage{amssymb, amscd, amsfonts,  amsthm}
\newtheorem{theorem}{Theorem}[section]

\theoremstyle{definition}

\newcommand{\remove}[1]{}

 \title[Finite and infinite primes]
{Finite and infinite primes in  models of PA}

\author{Daniele Mundici}
 
\address[D. Mundici]{Department of
Mathematics and Computer Science  ``Ulisse Dini'' \\
University of Florence\\
Viale Morgagni 67/A \\
I-50134 Florence \\
Italy}
\email{daniele.mundici@unifi.it }

\date{\today}

\begin{document}



 \begin{abstract}
We characterize  models of Peano arithmetic ($\mathsf{PA}$) 
with infinitely many {\it infinite} primes $p$
such that $p+2$ has no  {\it finite}  prime divisor.   
 \end{abstract}

\maketitle

\maketitle

\section{Statement of the main result}
We refer to \cite{chakei} for background in model theory.
The {\it standard} model
of first-order Peano arithmetic  ($\mathsf{PA}$)
 is  $\langle\mathbb N, +,\cdot, S,0\rangle$
where  $\mathbb N=\{0,1,2,\dots\}$, 
 $S$ is the successor function and $+,\cdot,0$
have their usual meaning,
\cite[1.4.11, p.42]{chakei}. Any other (nonisomorphic)
model is said to be {\it nonstandard}.

We assume that $\mathsf{PA}$ is consistent.
From G\"odel's second incompleteness theorem it
follows that the consistency of  $\mathsf{PA}$  can be
formalized as a sentence  $\phi$ in the language of  $\mathsf{PA}$,
but  neither $\phi$ nor $\neg \phi$ is a consequence of $\mathsf{PA}$.

A  {\it prime} in a model  $M$ of $\mathsf{PA}$
is an element $q \geq 2$ of the universe  $U$ of $M$  such that  
for all $x,y\in U$, if  $x\cdot y=q$  then either $x=q$ or
$y=q.$   
Any nonstandard model  $L$  of 
$\mathsf{PA}$ 
  has an initial segment $N_L$ isomorphic to the standard model.
  The primes of $L$ in $N_L$ are said to be {\it finite}.
  $L$ also has  an infinite set of {\it infinite} (also known as
{\it nonstandard}) primes.

Throughout, the symbol $p$
will be reserved for a finite prime, and the adjective ``finite''
will be  omitted when it is clear from the context. 
For  $k=1,2,\dots$ we let
\begin{equation}
\label{equation:lorena} 
p_k = \mbox{the $k$th prime in the list 2,3,5,7,11,\dots}\,.
\end{equation}

\medskip
We have not found in the literature a proof of the following result:
 
\begin{theorem}
\label{theorem:main} 
  There is a  countable  model $M$ of Peano arithmetic 
  containing an infinite sequence   $q_1<q_2<\cdots<q_i<\cdots$
  of infinite
  primes
such that  for each  $i=1,2,\dots,$\,\,  
no   finite  prime of $M$ is a divisor of $q_{i}+2$.
 \end{theorem}

\begin{proof}
Let  the  language $\mathcal L=\{+,\cdot, S, {0}\}$
of  $\mathsf{PA}$  be enriched  with 
new constant symbols  $c_1,c_2,\dots$\,.
Let  $\mathcal L^*=\mathcal L\cup \{c_1,c_2,\dots\}$
and $\mathcal L_n=\mathcal L \cup  \{c_1,c_2,\dots, c_n\}$.
By an  {\it $\mathcal L^*$-model of $\mathsf{PA}$}  
we mean a model $M$ of $\mathsf{PA}$ for the language $\mathcal L^*$,
where each constant symbol $c_i\in \mathcal L^*$ is interpreted
as an element $\bar c_i$ of the universe of $M$.
If the $\mathcal L$-reduct of $M$ is the standard
model of $\mathsf{PA}$ we say that $M$ is a
 {\it standard $\mathcal L^*$-model  of $\mathsf{PA}$.}
 Otherwise,  $M$ is a {\it nonstandard 
 $\mathcal L^*$-model  of $\mathsf{PA}$.}
   {\it Standard} and {\it nonstandard 
   $\mathcal L_n$-models  of $\mathsf{PA}$}  are similarly defined
   for the language $\mathcal L_n$.

To avoid confusion with the notation  $\equiv$ in modular
arithmetic,  we will let 
  = denote the identity relation symbol (denoted
$\equiv$ in \cite[1.3.7, p.25]{chakei}). 
While the relation symbol $\leq$
is not in $\mathcal L,$  we
 will  use the abbreviation   $x\leq y$ for $\exists z \,\, x+z=y$.
The derived relations $<,>,\geq$ are also freely used.

\medskip
\noindent
 Let $\alpha_1$ be the sentence
\begin{equation}
\label{equation:ascending-1}
c_1 \geq SS 0, 
\end{equation}
and for $i=2,3,4,\dots,$ let $\alpha_i$ be the sentence
\begin{equation}
\label{equation:ascending-n}
 c_{i}>c_{i-1}\ .
\end{equation}
Next  for   every $i=1,2,\dots$, let the sentence  $\beta_i$
be defined by
\begin{equation}
\label{equation:prime}
  \forall x\,\forall y \,\,\,\, x\cdot y=c_i\,\,\,
\to \,\,\,(x=c_i\,\, \vee\,\, y=c_i). 
\end{equation}
Finally,  for   every $i=1,2,\dots$
 and  prime $p=2,3,5,\dots$ 
let  $\gamma_{i,p}$ be the sentence
\begin{equation}
\label{equation:twin}
  \forall z\,\,\,\,c_i+SS 0\not=z\cdot
\underbrace{SS\cdots S}_{\mbox{\it \footnotesize p\,
\rm times}} 0.
\end{equation}
With  $\alpha_i, \beta_i, \gamma_{i,p}$  shorthand for the
sentences in \eqref{equation:ascending-1}-\eqref{equation:twin}, let   
\begin{equation}
\label{equation:omega-infinity}
 \Theta^*=\{\alpha_i, \beta_i, \gamma_{i,p}
 \mid i=1,2,\dots;  \,\,\, p=2,3,5,\dots \}.
\end{equation}
To prove the theorem we must construct  an
 $\mathcal L^*$-model of $\mathsf{PA}$
satisfying  $ \Theta^*$.

\smallskip
To this purpose, for   every $i=1,2,\dots$
 and prime $p=2,3,5,\dots$
 let   the sentence $\omega_{i,p}$ be  defined by 
\begin{equation}
\label{equation:omega}
\exists x \,\,\,\,\,\,c_i=x \cdot 
\underbrace{SS\cdots S}_{\mbox{\it \footnotesize p\,
\rm times}} 0 +
\underbrace{SS\cdots S}_{\mbox{\it \footnotesize p$-1$\,
\rm times}} 0.   
\end{equation}
Let  $N = \langle\mathbb N, +,\cdot, S,0, 
\bar c_1,\bar c_2,\dots \rangle$ be a standard
 $\mathcal L^*$-model of $\mathsf{PA}$
 with  $\bar c_i$   the interpretation  
of the constant symbol $c_i$.     
Then $N$ satisfies $ \omega_{i,p} $
if and only if  $\bar c_i$ is congruent to $p-1$ modulo $p$.
With the notation of \cite[1.3.14, p.28]{chakei} and \cite[p.29]{ireros}, 
\begin{equation}
\label{equation:modular}
N\models  \omega_{i,p} \,\,\,\,\,\, \mbox{if and only if}
\,\,\,\,\,\, \bar c_i\equiv p-1\,\,\,(p).
 \end{equation}
In particular,   from $N\models \omega_{i,p}$  it follows that   
 $\bar c_i+2\equiv 1\,\,\,(p)$,\, 
whence   $\bar c_i+2$ is not a multiple of $p$.
Recalling  \eqref{equation:twin} we may write 
$$
N\models \omega_{i,p} \to  \gamma_{i,p}.
$$
More generally, for each prime $p=2,3,5,\dots,$ the standard model  
$ \langle\mathbb N, +,\cdot, S,0\rangle$ satisfies the  sentence
$\sigma$ given by
 $$
\forall z[ (\exists x \,z=x \cdot 
\underbrace{SS\cdots S}_{\mbox{\it \footnotesize p\,
\rm times}} 0 +
\underbrace{SS\cdots S}_{\mbox{\it \footnotesize p$-1$\,
\rm times}} 0)\to
(\exists y \, z+SS 0=y \cdot 
\underbrace{SS\cdots S}_{\mbox{\it \footnotesize p\,
\rm times}} 0 +
S 0)].
$$
The sentence $\sigma$ is derivable  from the axioms of
$\mathsf{PA}$ and the rules of first-order logic,
$\mathsf{PA}\vdash \sigma.$
  \footnote{Most elementary number theory
  can be carried over in $\mathsf{PA}.$ 
Examples include the basic properties of
addition and multiplication and the
 uniqueness of remainder and quotient, 
from which $\mathsf{PA}\vdash \sigma$
routinely follows.
 For details  see, e.g.,  Theorem 1.3  in
  Van Oosten's lectures   ``Introduction to  Peano Arithmetic, G\"odel Incompleteness
  and Nonstandard Models", available online at
  https://webspace.science.uu.nl/$\sim$ooste110/syllabi/peanomoeder.pdf
    }\,\,\,
  Instantiation  $z\mapsto c_i$ now yields a proof in  
  $\mathsf{PA}$  of the
sentence  $ \omega_{i,p}\to  \gamma_{i,p}$.
In symbols, 
%
$\mathsf{PA}\vdash  \omega_{i,p}\to  \gamma_{i,p}$
 for  every $i=1,2,\dots$
 and prime $p=2,3,5,\dots$. 
  %
  %
Therefore, for every
 (standard or nonstandard) 
$\mathcal L^*$-model $M$ of
 $\mathsf{PA}$,
 \begin{equation}
\label{equation:to3}
M\models  \omega_{i,p}\to  \gamma_{i,p}\,\,\,\,
\,\,\,\mbox{for  every $i=1,2,\dots$
 and prime $p=2,3,5,\dots$.} 
\end{equation}

 \bigskip

\noindent
 {\rm  Claim:}
{\it Let  $n$ and $k$ be fixed but otherwise arbitrary
natural numbers $\geq 1$. 
  Let  $p_k$ the $k$th prime in the list    \eqref{equation:lorena}.
Then there exists a
standard $\mathcal L_n$-model 
   $N_{n,k}=\langle \mathbb N, +,\cdot,S, 0, \bar c_1,\dots,\bar c_n\rangle$
   of $\mathsf{PA}$ 
 such that
 \begin{equation}
 \label{equation:key}
 N_{n,k}\models
  \bigwedge_{i=1}^{n}    \alpha_{i}
  \wedge 
   \beta_i 
  \wedge   
\omega_{i,2}\wedge \omega_{i,3}\wedge \omega_{i,5}\wedge 
\dots\wedge  \omega_{i,p_k}  .\,
 \end{equation}
  In other words,   $\bar c_1<\dots<\bar c_n$ and for each
$i=1,\dots,n,\,\,\,$ $\bar c_i$ is a prime number 
such that 
$\bar c_i$ is congruent to $p-1$ modulo $p$ for any prime
 $p\leq p_k.$}

\bigskip
To prove this,  
for each $i=1,\dots,n$ let  the sentences
$\omega_{i,2},\omega_{i,3},\dots, \omega_{i,p_k}$
in \eqref{equation:omega}
be displayed as follows:
$$
\exists x   \,c_i=x\cdot SS 0+S 0,\,\,\,
\exists x   \,c_i =  x\cdot SSS 0+SS 0,\dots,
\exists x \,c_i=  x\cdot 
\underbrace{SS\cdots S}_{\mbox{\tiny $p_k$\,
\rm times}}\, 0 \,+
\underbrace{SS\cdots S}_{\mbox{\tiny $p_k-1$\,
\rm times}} 0.
$$
Recalling \eqref{equation:modular}, a
 standard $\mathcal L_n$-model  satisfies  the conjunction
$$
\bigwedge_{i=1}^n  \omega_{i,2}\wedge \omega_{i,3}
\wedge \omega_{i,5}\wedge \dots\wedge 
 \omega_{i,p_k}
$$
if and only if 
$\bar c_1,\dots,\bar c_n$ are solutions of the following
  system $\mathcal S$ of $n\times k$ modular equations,
  respectively
  in the unknowns $y_1,\dots, y_n\in \mathbb N$:
 $$
 \begin{array}{llll}
y_1\equiv  1   \,\,\,	\quad\quad  (2)
&\quad \quad y_2 \equiv  1   \,\,\,	\quad\quad  (2)
& \quad  \cdots
& \quad\quad y_n \equiv  1   \,\,\,	\quad\quad  (2)\\
y_1  \equiv  2   \,\,\,   \quad \quad(3)
&\quad\quad y_2\equiv  2   \,\,\,	  \quad\quad(3)
& \quad \cdots
& \quad\quad  y_n \equiv  2   \,\,\,	\quad\quad  (3)\\
y_1  \equiv  4    \,\,\,  \quad\quad(5)
&\quad\quad y_2\equiv  4   \,\,\,	 \quad\quad (5)
& \quad \cdots
& \quad\quad  y_n \equiv  4   \,\,\,	 \quad\quad (5)\\
\quad  \cdots   
&\quad \quad \quad  \cdots
&\quad    \cdots 
&\quad \quad \quad   \cdots\\  
y_1  \equiv  p_k-1  \,\,\, (p_k)
&\quad\quad y_2 \equiv  p_k-1    \,\,\,	 (p_k)
& \quad  \cdots
& \quad\quad  y_n \equiv  p_k-1    \,\,\,	 (p_k).
\end{array} 
$$

\noindent
The Chinese Remainder Theorem,
\cite[Theorem 1, p.34]{ireros},    provides
  infinitely many  solutions of
$\mathcal S$.
By direct inspection of the present special case, 
   the sequence
 $$
p_1p_2 \cdots p_k-1,\,\,\,\,\,\,\quad
2p_1p_2 \cdots p_k-1,\,\,\,\,\,\,\quad
3p_1p_2 \cdots p_k-1,\,\,\,\,\,\,\quad
4p_1p_2 \cdots p_k-1, \dots 
$$
yields all  solutions of $\mathcal S$.
 Dirichlet's theorem on arithmetic 
progressions \cite[\S 16, Theorem 1, p.251]{ireros}
ensures that this sequence contains infinitely many
primes.  {\it Thus we can   extract from it 
 an $n$-tuple  
 $\bar c_1<\dots<\bar c_n$ of   primes 
 solving the system $\mathcal S.$}
 We then obtain a standard $\mathcal L_n$-model 
   $N_{n,k}=\langle \mathbb N, +,\cdot,S, 0,
    \bar c_1,\dots,\bar c_n\rangle$
   of $\mathsf{PA}$ satisfying \eqref{equation:key}, as
   required to settle our claim.

\medskip
From    \eqref{equation:to3} and
\eqref{equation:key}, for all  $n,k = 1,2,\dots$    we obtain:
\begin{equation}
\label{equation:claim}
N_{n,k}
\models
   \bigwedge_{i=1}^{n} 
     \alpha_{i}
  \wedge 
   \beta_i 
  \wedge 
\gamma_{i,2}\wedge \gamma_{i,3}\wedge \gamma_{i,5}\wedge 
\dots\wedge  \gamma_{i,p_k}\,.
 \end{equation}
In other words,   $\bar c_1<\dots<\bar c_n$ and for each
$i=1,\dots,n,\,\,\,$ $\bar c_i$ is a prime number of $N_{n,k}$ 
such that $\bar c_i+2$ is not divisible by any prime $\leq p_k.$

\medskip
 
For every  $n,k=1,2,3,\dots$
 let us define the following subset of
the set 
$ \Theta^*$ of $\mathcal L^*$-sentences introduced  in 
 \eqref{equation:omega-infinity}:
 $$
\Theta_{n,k} =
  \{\alpha_i, \beta_i, \gamma_{i,p}
 \mid i=1,2,\dots,n;  \,\,\, p=2,3,5,\dots p_k \}.$$

\medskip
\noindent
Any finite subset  $\Theta$ of   $\Theta^*$
is contained in  $\Theta_{n,k}$ 
for all suitably large  $n$ and $k$.
From  \eqref{equation:claim} it follows that    
the standard  $\mathcal L_n$-model 
   $N_{n,k}$ 
satisfies  $\Theta.$
Therefore, for
 any finite subset  $\Pi$  of the set 
$\Pi^*$  of axioms of $\mathsf{PA},$\,\,\,
$N_{k,p}$  satisfies $\Pi\cup \Theta$.

Having thus shown that $ \Theta^*\cup\,\, \Pi^*$ is
finitely satisfiable, by G\"odel's compactness theorem
\cite[Theorem 1.3.22, p.67]{chakei} there is an  
an $\mathcal L^*$-model $M$  of $\mathsf{PA}$\
satisfying   $\Theta^*$.

By the L\"owenheim-Skolem  theorem
\cite[Corollaries 2.1.6-2.1.7, pp.67-68]{chakei},  
$M$ may be assumed  
countable.
For each $i=1,2,\dots,$ upon setting
$q_i=\bar c_i$,
the proof is complete.  
\end{proof}


\begin{thebibliography}{2}

\bibitem{chakei}
C.C.Chang, H.J.Keisler, Model Theory, 
Studies in Logic and the Foundations of Mathematics, Vol. 73, 
Third impression.  Elsevier, Amsterdam, 1992. 

\bibitem{ireros}
K.Ireland, M.Rosen, A Classical Introduction
to Modern Number Theory, Second edition, fifth corrected
printing.
Graduate Texts in
Mathematics, Vol. 84, Springer-Verlag, New York,  Heidelberg,  Berlin,
1990.

\end{thebibliography}
\end{document}